\documentclass[12pt]{article}

\usepackage{mathrsfs}
\usepackage{amssymb}
\usepackage{amsmath}
\usepackage{amsfonts}
\usepackage{amstext}
\usepackage{amsthm}
\usepackage{graphicx}
\usepackage{subfig}
\usepackage{color}
\usepackage{indentfirst,latexsym,bm}
\usepackage{mathrsfs}
\usepackage{cite}
\usepackage[all]{xy}
\usepackage{newlfont}

\usepackage{enumerate}
\numberwithin{equation}{section}

\newtheorem{theorem}{Theorem}[section]

\newtheorem{lemma}[theorem]{Lemma}

\numberwithin{equation}{section}

\begin{document}

\title{Nontrivial domains for the Bernoulli's overdetermined problem with small constant vorticity
\thanks{
The research was supported by National Natural Science Foundation of
China (No. 12301133), the Postdoctoral Science Foundation of China (No. 2023M741441, No. 2024T170353) and Jiangsu Education Department (No. 23KJB110007). }}
\author{
Yong Zhang, Jialong Zhu
\thanks{Yong Zhang. School of Mathematical Sciences, Jiangsu University, Zhenjiang, 212013, P.R. China. E-mail: zhangyong@ujs.edu.cn}
}

\date{}
\maketitle

\renewcommand{\abstractname}{Abstract}

\begin{abstract}
In this paper, we first apply the implicit function theorem to find nontrivial solutions to the following Bernoulli's overdetermined problem
\begin{equation}
-\Delta u=\gamma\,\, \text{in}\,\,\Omega \setminus\overline{D}, \,\, u=1,\,\,\partial_\mathbf{n} u=Q\,\,\text{on}\,\,\partial D, \,\,u=0\,\,\text{on}\,\,\partial\Omega \nonumber
\end{equation}
for a fixed boundary $\partial\Omega$ and a free boundary $\partial D$ in $\mathbb{R}^{N}$ with $N\geq2$, where $\gamma$ is a sufficiently small constant, the Bernoulli value $Q$ is determined by the construction, and $\Omega \setminus\overline{D}$ is a bounded ring-shaped domain. In dimension $N=n+1$ with $n\geq2$, we further study the existence of a nontrivial unbounded domain $\Omega \setminus\overline{D}$ in $\mathbb{R}^N$ such that the Bernoulli's overdetermined problem admits a solution. This unbounded domain we construct is periodic in one variable and radial in the other variables by bifurcation analysis.
\end{abstract}

\emph{Keywords:} Bernoulli's overdetermined problem; free boundary; nontrivial domains
\\

\emph{AMS Subjection Classification(2010):} 35B32; 35N05; 51M10


\section{Introduction and main results}

For the bounded problem, assume that $\Omega\subset\mathbb{R}^N$, $N\geq2$, is a bounded domain and let $D\Subset\Omega$ be a subdomain with smooth boundary. We consider the following Bernoulli free boundary problem with fixed constant $\gamma$:
\begin{equation}\label{e1.1}
\left\{
\begin{array}{ll}
-\Delta u=\gamma &\text{in }\Omega\setminus\overline D,\\
u=0 &\text{on }\partial\Omega,\\
u=1 &\text{on }\partial D,\\
\dfrac{\partial u}{\partial\mathbf n}=Q &\text{on }\partial D.
\end{array}
\right.
\end{equation}
Here $\mathbf n$ is the outer unit normal to $\Omega\setminus\overline D$ and $Q>0$ is constant along the inner boundary. In our constructions, $Q$ is an output of the problem. For a given $\Omega$, we seek a nearby inner domain $D$ and a constant $Q$ for which \eqref{e1.1} is solvable. The case $\gamma=0$ is the classical irrotational Bernoulli problem. In the plane, $u$ is the stream function of a stationary incompressible inviscid flow with constant vorticity circulating around a bubble $D$ in a container $\Omega$. Since $u$ is constant on $\partial D$ and $Q>0$, the additional condition gives
\[
\partial_{\mathbf n}u=\nabla u\cdot\mathbf n=|\nabla u|=Q,
\]
so the kinetic energy is constant along $\partial D$.

For $\gamma=0$, the symmetry results of Alessandrini \cite{Alessandrini} and Reichel \cite{Reichel} imply that, if $\Omega=B_1$, then $D$ is a concentric ball. For constant $\gamma$, we use the same reference pair
\[
(D_0,\Omega_0)=(B_R,B_1),\qquad 0<R<1.
\]
The radial solution on this annulus is
\begin{equation}\label{e1.2}
u_R(r)=
\begin{cases}
\dfrac{\gamma}{2N}(1-r^2)+a_R(r^{2-N}-1),&N\geq3,\\[1ex]
\dfrac{\gamma}{4}(1-r^2)+a_R\log r,&N=2,
\end{cases}
\quad
a_R=
\begin{cases}
\dfrac{1+\frac{\gamma}{2N}(R^2-1)}{R^{2-N}-1},&N\geq3,\\[1ex]
\dfrac{1+\frac{\gamma}{4}(R^2-1)}{\log R},&N=2.
\end{cases}
\end{equation}
We set
\begin{equation}\label{e1.3}
Q_R^{\mathrm{bd}}:=-u_R'(R),\qquad
\kappa_R^{\mathrm{bd}}:=u_R''(R)
=\frac{N-1}{R}Q_R^{\mathrm{bd}}-\gamma.
\end{equation}
For each fixed $R$, $Q_R^{\mathrm{bd}}>0$ if $|\gamma|$ is sufficiently small. When $\gamma=0$, \eqref{e1.2} reduces to the usual harmonic annular profile, with
\[
Q_R^{\mathrm{bd}}=
\begin{cases}
\dfrac{(N-2)R^{1-N}}{R^{2-N}-1},&N\geq3,\\[1ex]
-\dfrac{1}{R\log R},&N=2.
\end{cases}
\]

Fix $\alpha\in(0,1)$. Now we suppose that $(D_f,\Omega_g)$ is the perturbed domains of $(D_0,\Omega_0)$, where $f\in \mathcal{F}$ and $g\in \mathcal{G}$ with
\begin{equation}\label{e1.4}
\mathcal{F}=\{f\in C^{2,\alpha}(\partial D_0) : \int_{\partial D_0}f=0\} \quad \text{and} \quad \mathcal{G}=\{g\in C^{2,\alpha}(\partial \Omega_0) : \int_{\partial \Omega_0}g=0\}.
\end{equation}
If the functions $f$ and $g$ are sufficiently small in $C^{2,\alpha}$, the perturbed domains $D_f$ and $\Omega_g$ are well defined as the unique bounded domains whose boundaries are
\begin{equation}\label{e1.5}
\partial D_f=\{x+f(x)\mathbf{n} : x\in \partial D_0\} \quad  \text{and} \quad \partial \Omega_g=\{x+g(x)\mathbf{n} : x\in \partial \Omega_0\}.
\end{equation}
On $\partial D_0$, it is easy to check that $\mathbf n=-x/R$, while on $\partial\Omega_0$, $\mathbf n=x$. Thus a positive $f$ moves the inner boundary inward: $\partial D_f=\{(R-f(R\theta))\theta:\theta\in\mathbb S^{N-1}\}$. We retain this normal-displacement convention throughout the bounded case.

In this paper, we are mainly concerned with the following two issues. One is to construct an inner boundary for each sufficiently small perturbation $\Omega_g$ of $B_1$. If $\Omega_g$ is not a ball, the admitting domain $\Omega_g\setminus\overline{D_f}$ is not a standard annulus. The so-called {\bf admitting domain} means that the region which makes problem (\ref{e1.1}) solvable. Of course, this nontrivial domain $\Omega_g \setminus\overline{D_f}$ is bounded. The other aim is to find some new unbounded domains admitting a solution of (\ref{e1.1}). The unbounded domains we will construct are periodic in one variable and radial in the other variables, and they bifurcate from a straight hollow cylinder.
Our main results can be concluded as follows:
\begin{theorem}\label{thm1.1}
Fix $N\geq2$, $R\in(0,1)$ and $\alpha\in(0,1)$. There exists $\gamma_{\mathrm{bd}}(R,N)>0$ such that, for every fixed $|\gamma|<\gamma_{\mathrm{bd}}(R,N)$ and every sufficiently small $g\in\mathcal G$, there is a unique small $f=f(g)\in\mathcal F$ for which $\Omega_g\setminus\overline{D_f}$ admits a solution of \eqref{e1.1} for some constant $Q=Q(g)>0$. The map $g\mapsto f(g)$ is smooth, and uniqueness is understood in a neighborhood of $(0,0)\in\mathcal F\times\mathcal G$. Moreover, $f(0)=0$, $Q(0)=Q_R^{\mathrm{bd}}$, and the reference solution is given by \eqref{e1.2}.

\begin{figure}[ht]
\centering
\includegraphics[width=0.7\textwidth]{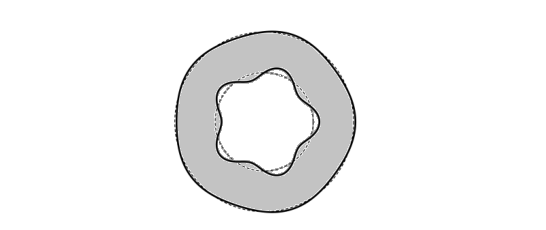}
\caption{The nontrivial bounded domain $\Omega_g \setminus\overline{D_f}$.}
\end{figure}
\end{theorem}

We next consider a fixed outer cylinder and a free inner boundary. Put $N=n+1$ with $n\geq2$, and define
\[
\Omega_1^t=\{(x,t)\in\mathbb R^n\times\mathbb R:|x|<1\},\qquad
D_\phi^t=\{(x,t):|x|<\phi(t)\},
\]
where $\phi(t)=R+\varphi(t)$ is even and $2\pi$-periodic, with $0<\phi(t)<1$ for every $t$. The perturbation $\varphi$ can have either sign. Let $U_R$ be the unique radial solution of
\[
-U_R''-\frac{n-1}{r}U_R'=\gamma\quad(R<r<1),\qquad
U_R(R)=1,\quad U_R(1)=0,
\]
and set
\[
Q_R^{\mathrm{cyl}}=-U_R'(R),\qquad
\kappa_R^{\mathrm{cyl}}=U_R''(R)=\frac{n-1}{R}Q_R^{\mathrm{cyl}}-\gamma.
\]
The explicit formula for $U_R$ is given in \eqref{eq3.6}. The cross-sectional dimension here is $n$, rather than the ambient dimension $N$.

\begin{theorem}\label{thm1.2}
For each $n\geq2$ and $\alpha\in(0,1)$, there exists $\gamma_{\mathrm{cyl}}(n)>0$ such that, for every fixed $|\gamma|<\gamma_{\mathrm{cyl}}(n)$, there are $R_*=R_*(n,\gamma)\in(0,1)$ and a smooth map
\[
(-\varepsilon,\varepsilon)\longrightarrow
(0,1)\times C^{2,\alpha}_{p,e}(\mathbb R),\qquad
s\longmapsto(R_s,\varphi_s),
\]
where $C^{2,\alpha}_{p,e}$ denotes the even $2\pi$-periodic H\"older space, such that $R_0=R_*$, $\varphi_0=0$, and $0<\phi_s:=R_s+\varphi_s<1$. For every $s$, there is an even, $2\pi$-periodic solution $u_s\in C^{2,\alpha}(\overline{\Omega_1^t\setminus\overline{D_{\phi_s}^t}})$ of
\begin{equation}\label{e1.6}
\left\{
\begin{array}{ll}
-\Delta u_s=\gamma &\text{in }\Omega_1^t\setminus\overline{D_{\phi_s}^t},\\
u_s=0 &\text{on }\partial\Omega_1^t,\\
u_s=1 &\text{on }\partial D_{\phi_s}^t,\\
\partial_{\mathbf n}u_s=Q_{R_s}^{\mathrm{cyl}}>0
&\text{on }\partial D_{\phi_s}^t.
\end{array}
\right.
\end{equation}
in the domain
\begin{equation}\label{e1.7}
\Omega_1^t\setminus\overline{D_{\phi_s}^t}
=\{(x,t)\in\mathbb R^n\times\mathbb R:\phi_s(t)<|x|<1\}.
\end{equation}
Moreover,
\[
\varphi_s(t)=s\cos t+s\mu_s(t),\qquad
\int_0^{2\pi}\mu_s(t)\cos t\,dt=0,\qquad \mu_0=0,
\]
where $s\mapsto\mu_s\in C^{2,\alpha}_{p,e}(\mathbb R)$ is smooth. In particular, $\phi_s$ is nonconstant for $s\ne0$. No zero-average condition is imposed on $\varphi_s$.

\begin{figure}[ht]
\centering
\includegraphics[width=0.7\textwidth]{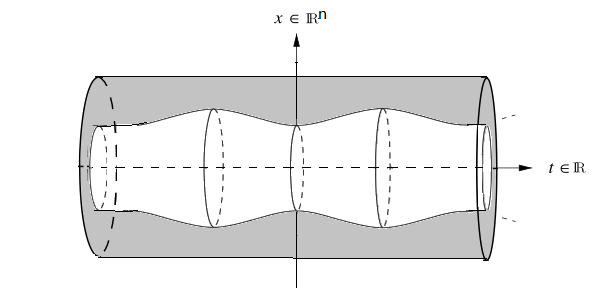}
\caption{The nontrivial unbounded domain $\Omega^{t}_{1}\setminus \overline{D^{t}_{\phi_s}}$.}
\end{figure}
\end{theorem}

The rest of this paper is arranged as follows. In Section 2, we prove the bounded perturbation result by an implicit function theorem. In Section 3, we compute the cylindrical linearization, establish its spectral and Fredholm properties, and prove the periodic bifurcation result.

\section{Bounded case}

\subsection{Function spaces and abstract operators}
In this subsection, we introduce the operator for the bounded problem. Let
\begin{equation}\label{e2.1}
\mathcal H=\left\{h\in C^{1,\alpha}(\partial D_0):
\int_{\partial D_0}h=0\right\}.
\end{equation}
For sufficiently small $(f,g)\in\mathcal F\times\mathcal G$, let $u_{f,g}$ be the unique solution of
\begin{equation}\label{e2.2}
\left\{
\begin{array}{ll}
-\Delta u_{f,g}=\gamma &\text{in }\Omega_g\setminus\overline{D_f},\\
u_{f,g}=0 &\text{on }\partial\Omega_g,\\
u_{f,g}=1 &\text{on }\partial D_f.
\end{array}
\right.
\end{equation}
In a neighborhood of $(0,0)$, we define
\begin{equation}\label{e2.3}
\Psi:\mathcal F\times\mathcal G\longrightarrow\mathcal H,\qquad
\Psi(f,g)=q_{f,g}-\frac1{|\partial D_0|}\int_{\partial D_0}q_{f,g},
\end{equation}
where the pulled-back Neumann trace is
\begin{equation}\label{e2.4}
q_{f,g}(x)=\nabla u_{f,g}(x+f(x)\mathbf n(x))
\cdot\mathbf n_{f,g}(x+f(x)\mathbf n(x)),\qquad x\in\partial D_0.
\end{equation}
Here $\mathbf n_{f,g}$ is the outer unit normal to the perturbed annular region. Since the Dirichlet value is constant on the inner boundary and $Q_R^{\mathrm{bd}}>0$, we have that
\begin{equation}\label{e2.5}
\mathbf n_{f,g}=\frac{\nabla u_{f,g}}{|\nabla u_{f,g}|}
\quad\text{on }\partial D_f.
\end{equation}
Thus $\Psi(f,g)=0$ means exactly that \eqref{e1.1} holds with the positive constant
\[
Q=\frac1{|\partial D_0|}\int_{\partial D_0}q_{f,g}.
\]

\subsection{shape derivatives}
For $f_0\in\mathcal F$ and small $t$, set $u_t=u_{tf_0,0}$ and $u=u_R$. For each fixed $x$ in the reference annulus $A_0=\Omega_0\setminus\overline{D_0}$, $u_t(x)$ is defined for all sufficiently small $|t|$. Its Eulerian derivative is
\[
u_f'(x)=\left.\frac{d}{dt}\right|_{t=0}u_t(x).
\]
To obtain boundary traces, choose a $C^{2,\alpha}$ vector field $\mathbf p$ extending $f_0\mathbf n$ on $\partial D_0$ and vanishing near $\partial\Omega_0$, and pull the solutions back by $\Theta_t=\mathrm{Id}+t\mathbf p$. The material derivative is $\dot u_f=\left.\frac{d}{dt}\right|_{t=0}(u_t\circ\Theta_t)$, and the shape derivative satisfies \cite{HenrotP}
\begin{equation}\label{e2.6}
u_f'=\dot u_f-\nabla u\cdot\mathbf p.
\end{equation}
Because $\gamma$ is fixed, differentiation of the equation gives a harmonic shape derivative. Differentiating the constant Dirichlet data on both boundaries gives
\begin{equation}\label{e2.7}
\left\{
\begin{array}{ll}
-\Delta u_f'=0 &\text{in }A_0,\\
u_f'=0 &\text{on }\partial\Omega_0,\\
u_f'=-Q_R^{\mathrm{bd}}f_0 &\text{on }\partial D_0.
\end{array}
\right.
\end{equation}
Similarly, an outer displacement $g_0\mathbf n$ gives the boundary value $u_g'=-\partial_{\mathbf n}u\,g_0$ on $\partial\Omega_0$ and $u_g'=0$ on $\partial D_0$.

\subsection{Spherical harmonics}
The real spherical harmonics of degree $k\geq 0$ are the eigenfunctions of the negative spherical
Laplace-Beltrami operator $-\Delta_{\mathbb{S}^{N-1}}$ corresponding to the eigenvalue $\lambda_k=k(N+k-2)$, that is,
$$
-\Delta_{\mathbb{S}^{N-1}} Y^i_k=\lambda_k Y^i_k.
$$
Write $d_{k,N}$ for the multiplicity of $\lambda_k$, then
\[
d_{k,N}=\frac{(2k+N-2)(k+N-3)!}{k!(N-2)!} \quad \text{for}~~N\geq3,
\]
while $d_{0,2}=1$ and $d_{k,2}=2$ for $k\geq1$. Choose an orthonormal basis $\{Y_k^i:1\leq i\leq d_{k,N}\}$ in each eigenspace. These bases together form a complete orthonormal basis of $L^2(\mathbb S^{N-1})$. We identify functions on $\partial D_0$ with functions of $\theta$ by $x=R\theta$.

Thus, for any $u\in L^2(\mathbb{S}^{N-1})$, we have that
\begin{equation}\label{e2.8}
u=\sum_{k\geq 0}\sum_{i=1}^{d_{k,N}}\alpha_k^i Y_k^i, \quad \alpha_k^i=\int_{\mathbb{S}^{N-1}}u  Y_k^i~d\mathbb{S}^{N-1}.
\end{equation}
For each $k\geq 0$ we define the linear subspace
\begin{equation}\label{e2.9}
\mathbf{Y}_{k}:=\text{span}\{Y_k^i:1\leq i\leq d_{k,N}\}
\end{equation}
with any $u\in \mathbf{Y}_{k}$ being such that $-\Delta_{\mathbb{S}^{N-1}}u=\lambda_k u$.
In addition, for any $s\geq0$, we can define the inhomogeneous Sobolev space $H^s(\mathbb{S}^{N-1})$. More precisely, we set
\begin{equation}\label{e2.10}
H^s(\mathbb{S}^{N-1}):=\{u\in L^2(\mathbb{S}^{N-1}): u=\sum_{k\geq 0}\sum_{i=1}^{d_{k,N}}\alpha_k^i Y_k^i,\quad \sum_{k\geq 0}\sum_{i=1}^{d_{k,N}}(1+k^2)^s(\alpha_k^i)^2<\infty\}.
\end{equation}
Notice that the eigenspace corresponding to
the eigenvalue $\lambda_0=0$ is the one-dimensional space of constant functions on $\mathbb{S}^{N-1}$. Moreover, such constant function term does not appear in the expansion of a function
of zero average on $\mathbb{S}^{N-1}$.

\subsection{The Proof of Theorem \ref{thm1.1}}
For any fixed $g\in\mathcal G$, zeros of $\Psi(\cdot,g)$ correspond to solutions of \eqref{e1.1} with a constant $Q$ determined by the inner Neumann trace. We therefore seek small $f\in\mathcal F$ such that
\begin{equation}\label{e3.1}
\Psi(f,g)=0.
\end{equation}
Since $u_R$ is radial, $\Psi(0,0)=0$.

\begin{theorem}\label{thm:bounded-linearization}
The map $\Psi$ is smooth near $(0,0)\in\mathcal F\times\mathcal G$, and
\begin{equation}\label{e3.2}
\partial_f\Psi(0,0)[f_0]=\partial_{\mathbf n}u_f'
+\kappa_R^{\mathrm{bd}}f_0,\qquad f_0\in\mathcal F,
\end{equation}
where $u_f'$ solves \eqref{e2.7}.
\end{theorem}
\begin{proof}
Extend the normal graphs to a $C^{2,\alpha}$ diffeomorphism $\Theta_{f,g}$ of the reference annulus. The pull-back $u_{f,g}\circ\Theta_{f,g}$ solves a uniformly elliptic Dirichlet problem on $A_0$ whose coefficients depend smoothly on $(f,g)$. The Dirichlet operator is an isomorphism from the space of $C^{2,\alpha}$ functions with zero boundary values to $C^{0,\alpha}$. Smooth dependence of its inverse and Schauder estimates give a smooth solution map into $C^{2,\alpha}$, and a smooth pulled-back Neumann map into $C^{1,\alpha}$; see \cite{Gilbarg,HenrotP}.

Set $u_t=u_{tf_0,0}$. At the reference boundary, the derivative of the unit normal is tangential and $\nabla u_R$ is normal. Hence differentiating \eqref{e2.4}, we can obtain that  
\begin{equation}\label{e3.5}
\begin{aligned}
\left.\frac{d}{dt}\right|_{t=0}q_{tf_0,0}
=\left.\frac{d}{dt}\right|_{t=0}\nabla u_{tf_0,0}(x+tf_0\mathbf n)\cdot \mathbf n
&=\left(\nabla u_f'+D^2u_R f_0 \mathbf n \right)\cdot \mathbf n \\
=\partial_{\mathbf n}u_f'+u''_{R}f_0=\partial_{\mathbf n}u_f'+\kappa_R^{\mathrm{bd}}f_0,
\end{aligned}
\end{equation}
where $u'_f$ is the solution of \eqref{e2.7} (see also \cite[Theorem 3.3]{Cav}).
Since $\int_{\partial D_0}f_0=0$, the spherical average of $u_f'$ solves the radial harmonic equation with zero values at both boundaries. Therefore, $\int_{\partial D_0}\partial_{\mathbf n}u_f'=0$, and the average of the right-hand side of \eqref{e3.5} vanishes. It follows from \eqref{e2.3} and \eqref{e3.5}
that 
$$
\partial_f\Psi(0,0)[f_0]=\left.\frac{d}{dt}\right|_{t=0}q_{tf_0,0}=\partial_{\mathbf n}u_f'+\kappa_R^{\mathrm{bd}}f_0.
$$
\end{proof}

We now compute the multipliers by spherical harmonics.
\begin{theorem}\label{thm:bounded-expansion}
Let $f_0\in\mathcal F$, with expansion
\begin{equation}\label{e3.7}
f_0(R\theta)=\sum_{k\geq1}\sum_{i=1}^{d_{k,N}}\alpha_k^iY_k^i(\theta).
\end{equation}
Then the solution of \eqref{e2.7} has the $L^2$ spherical harmonic expansion
\[
u_f'(r,\theta)=\sum_{k\geq1}\sum_{i=1}^{d_{k,N}}
(A_kr^k+B_kr^{2-N-k})\alpha_k^iY_k^i(\theta),
\]
where, for all $N\geq2$ and $k\geq1$,
\[
A_k=-\frac{Q_R^{\mathrm{bd}}}{R^k-R^{2-N-k}},\qquad
B_k=\frac{Q_R^{\mathrm{bd}}}{R^k-R^{2-N-k}}.
\]
\end{theorem}
\begin{proof}
We first derive the equation satisfied by each spherical harmonic coefficient of the solution. Write $v=u_f'$ and, for $k\geq0$ and $1\leq i\leq d_{k,N}$, define
\[
v_k^i(r)=\int_{\mathbb S^{N-1}}v(r,\theta)Y_k^i(\theta)\,d\theta,
\qquad R\leq r\leq1.
\]
The solution $v$ belongs to $C^{2,\alpha}(\overline{A_0})$, so we can differentiate these coefficients twice with respect to $r$. In polar coordinates,
\[
\Delta v=\partial_{rr}v+\frac{N-1}{r}\partial_rv
+\frac1{r^2}\Delta_{\mathbb S^{N-1}}v.
\]
We multiply the equation $\Delta v=0$ by $Y_k^i$ and integrate over the sphere. Integration by parts on $\mathbb S^{N-1}$ gives
\[
\int_{\mathbb S^{N-1}}\Delta_{\mathbb S^{N-1}}v\,Y_k^i\,d\theta
=\int_{\mathbb S^{N-1}}v\,\Delta_{\mathbb S^{N-1}}Y_k^i\,d\theta
=-k(k+N-2)v_k^i(r).
\]
Thus each coefficient satisfies
\[
(v_k^i)''+\frac{N-1}{r}(v_k^i)'
-\frac{k(k+N-2)}{r^2}v_k^i=0.
\]
The boundary conditions in \eqref{e2.7} imply
\[
v_k^i(1)=0,\qquad
v_k^i(R)=-Q_R^{\mathrm{bd}}\alpha_k^i
\quad(k\geq1).
\]
The constant coefficient has zero data at both endpoints. Indeed,
\[
\int_{\mathbb S^{N-1}}f_0(R\theta)\,d\theta
=R^{1-N}\int_{\partial D_0}f_0=0.
\]
It follows that $v_0^1=0$. This explains why the expansion starts from $k=1$.

For $k\geq1$, we find the radial factor $S_k$ with inner value $-Q_R^{\mathrm{bd}}$ and outer value zero. It solves
\begin{equation}\label{e3.8}
S_k''+\frac{N-1}{r}S_k'
-\frac{k(k+N-2)}{r^2}S_k=0,
\qquad R<r<1.
\end{equation}
Since this is an Euler equation, we substitute $S_k=r^a$. The characteristic equation is
\[
a(a-1)+(N-1)a-k(k+N-2)
=(a-k)(a+k+N-2)=0.
\]
Its roots are $k$ and $2-N-k$. They are distinct because $2k+N-2>0$. Therefore
\[
S_k(r)=A_kr^k+B_kr^{2-N-k}.
\]
The two boundary conditions give the linear system
\[
A_k+B_k=0,\qquad
A_kR^k+B_kR^{2-N-k}=-Q_R^{\mathrm{bd}}.
\]
We first obtain $B_k=-A_k$ from the outer condition. Substitution into the inner condition yields
\[
A_k=-\frac{Q_R^{\mathrm{bd}}}{R^k-R^{2-N-k}},\qquad
B_k=\frac{Q_R^{\mathrm{bd}}}{R^k-R^{2-N-k}}.
\]
The denominator is nonzero: with $m_k=2k+N-2>0$,
\[
R^k-R^{2-N-k}=R^{2-N-k}(R^{m_k}-1)<0.
\]
Equivalently,
\[
S_k(r)=-Q_R^{\mathrm{bd}}
\frac{r^k-r^{2-N-k}}{R^k-R^{2-N-k}}.
\]
This formula also applies when $N=2$ and $k\geq1$: the two radial powers are $r^k$ and $r^{-k}$. The repeated-root case, with radial solutions $1$ and $\log r$, occurs only when $N=2$ and $k=0$, whose coefficient has already been shown to vanish.

Uniqueness for the radial boundary value problem now gives $v_k^i=\alpha_k^iS_k$. For each fixed $r\in[R,1]$, completeness of the spherical harmonics reconstructs $v(r,\cdot)$ in $L^2(\mathbb S^{N-1})$ and proves the stated expansion. The same projection calculation applied to the radial derivative gives
\[
(v_k^i)'(r)=\alpha_k^iA_k
\big(k r^{k-1}+(N+k-2)r^{1-N-k}\big).
\]
In particular, because the fluid normal on the inner sphere is $\mathbf n=-\partial_r$, the coefficients of its Neumann trace are
\[
\int_{\mathbb S^{N-1}}\partial_{\mathbf n}v(R,\theta)
Y_k^i(\theta)\,d\theta
=-\alpha_k^iA_k
\big(kR^{k-1}+(N+k-2)R^{1-N-k}\big).
\]
These are the coefficients used in the linearized boundary operator below.
\end{proof}

Combining Theorems~\ref{thm:bounded-linearization} and \ref{thm:bounded-expansion}, we obtain
\begin{equation}\label{e3.9}
\partial_f\Psi(0,0)[f_0]
=\sum_{k\geq1}\sum_{i=1}^{d_{k,N}}\sigma_k\alpha_k^iY_k^i,
\end{equation}
where, with $m_k=2k+N-2$,
\begin{equation}\label{e3.10}
\begin{aligned}
\sigma_k
&=\kappa_R^{\mathrm{bd}}
-A_k\big(kR^{k-1}+(N+k-2)R^{1-N-k}\big)\\
&=-Q_R^{\mathrm{bd}}\delta_k(R)-\gamma,\qquad
\delta_k(R)=\frac{k-1+(k+N-1)R^{m_k}}{R(1-R^{m_k})}>0.
\end{aligned}
\end{equation}
Here the minus sign in the normal derivative is essential because $\mathbf n=-\partial_r$ on $\partial D_0$. In particular, the multipliers are negative at $\gamma=0$. They remain nonzero uniformly in $k$ for small $|\gamma|$. Indeed,
\[
\delta_1(R)=\frac{NR^{N-1}}{1-R^N},\qquad
\delta_k(R)\geq\frac1R\quad(k\geq2).
\]
Let $\delta_*:=\min\{\delta_1(R),1/R\}>0$ and $Q_0:=Q_R^{\mathrm{bd}}|_{\gamma=0}>0$. Choose $\gamma_{\mathrm{bd}}(R,N)>0$ so small that $Q_R^{\mathrm{bd}}\geq Q_0/2$ and $|\gamma|<Q_0\delta_*/4$ throughout the stated interval. Then
\[
\sigma_k\leq-\frac{Q_0\delta_*}{4}<0\qquad(k\geq1).
\]

\begin{proof}[Proof of Theorem~\ref{thm1.1}]
Let $T$ be the inner Dirichlet-to-Neumann map for harmonic functions on $A_0$ with zero outer Dirichlet data, acting from $C^{2,\alpha}(\partial D_0)$ to $C^{1,\alpha}(\partial D_0)$. This map is an isomorphism: for any inner Neumann datum $h\in C^{1,\alpha}$, the mixed problem
\begin{equation}\label{e3.12}
\left\{
\begin{array}{ll}
-\Delta w=0 &\text{in }A_0,\\
w=0 &\text{on }\partial\Omega_0,\\
\partial_{\mathbf n}w=h &\text{on }\partial D_0
\end{array}
\right.
\end{equation}
has a unique $C^{2,\alpha}$ solution and a Schauder estimate, since the two boundary components are disjoint and the outer Dirichlet condition removes the constant ambiguity \cite{Gilbarg}. Its inner trace is $T^{-1}h$.

Moreover $T:\mathcal F\to\mathcal H$ is also an isomorphism. Spherical averaging shows that mean-zero inner Dirichlet data give mean-zero inner Neumann data. Conversely, if $h$ has zero mean in \eqref{e3.12}, the spherical average of $w$ has zero outer value and zero inner normal derivative, and is identically zero; thus its inner trace has zero mean.

By \eqref{e2.7} and \eqref{e3.2},
\[
\partial_f\Psi(0,0)=-Q_R^{\mathrm{bd}}T+\kappa_R^{\mathrm{bd}}J,
\]
where $J:\mathcal F\to\mathcal H$ is the compact inclusion $C^{2,\alpha}\hookrightarrow C^{1,\alpha}$. The operator is therefore Fredholm of index zero. Formula \eqref{e3.10} shows that its kernel is zero, so it is an isomorphism. This argument establishes the H\"older inverse directly; nonzero multipliers alone would not justify an inverse between these spaces.

The implicit function theorem now gives a unique smooth map $g\mapsto f(g)\in\mathcal F$ solving \eqref{e3.1}, with $f(0)=0$. The resulting constant is
\[
Q(g)=\frac1{|\partial D_0|}\int_{\partial D_0}q_{f(g),g}.
\]
It is positive for small $g$ by continuity from $Q_R^{\mathrm{bd}}>0$. This proves the theorem.
\end{proof}

\section{Unbounded case}

\subsection{The transformed problem}
Let
\[
C^{k,\alpha}_{p,e}(\mathbb R)
=\{\psi\in C^{k,\alpha}(\mathbb R):\psi(-t)=\psi(t),\
\psi(t+2\pi)=\psi(t)\},
\]
with its usual periodic H\"older norm, and set $X=C^{2,\alpha}_{p,e}(\mathbb R)$ and $Y=C^{1,\alpha}_{p,e}(\mathbb R)$. We retain the constant mode in both spaces. Write
\[
A_\phi=\Omega_1^t\setminus\overline{D_\phi^t}
=\{(x,t):\phi(t)<|x|<1\},\qquad A_R=A_\phi\big|_{\phi\equiv R}.
\]
For $\phi=R+\varphi$ near a fixed $R\in(0,1)$, choose a smooth radial cutoff $\eta$ which equals $1/R$ near $|x|=R$ and vanishes near $|x|=1$. For sufficiently small $\|\varphi\|_{C^{2,\alpha}}$, the map
\[
\Phi_\phi:A_R\longrightarrow A_\phi,\qquad
\Phi_\phi(x,t)=\big((1+\eta(x)\varphi(t))x,t\big)
\]
is a $C^{2,\alpha}$ diffeomorphism. Let $g_\phi=\Phi_\phi^*g_e$. On the inner collar, where $\eta=1/R$,
\[
g_\phi=(1+\varphi/R)^2\,dx\cdot dx
+\big(1+|x|^2(\varphi'/R)^2\big)\,dt^2
+2(1+\varphi/R)(\varphi'/R)(x\cdot dx)\,dt.
\]
Near the outer boundary $g_\phi=g_e$, and at the reference profile,
\begin{equation}\label{eq3.1}
\Delta_{g_R}=\Delta_x+\partial_{tt}.
\end{equation}
The physical Dirichlet solution pulled back to $A_R$ solves
\begin{equation}\label{eq3.2}
\left\{
\begin{array}{ll}
-\Delta_{g_\phi}u_\phi=\gamma &\text{in }A_R,\\
u_\phi=0 &\text{on }|x|=1,\\
u_\phi=1 &\text{on }|x|=R.
\end{array}
\right.
\end{equation}
The additional Neumann condition, with $R$ serving as the bifurcation parameter, is
\begin{equation}\label{eq3.3}
\partial_{\nu_\phi}u_\phi=Q_R^{\mathrm{cyl}}
\quad\text{on }|x|=R,
\end{equation}
where $\nu_\phi$ is the outer $g_\phi$-unit normal to the fluid region. Since $\Phi_\phi$ is an isometry,
\begin{equation}\label{eq3.4}
d\Phi_\phi\,\nu_\phi=\mathbf n_\phi\circ\Phi_\phi,
\end{equation}
and the physical normal is
\begin{equation}\label{eq3.5}
\mathbf n_\phi(x,t)=\frac{(-x/|x|,\phi'(t))}{\sqrt{1+\phi'(t)^2}}
\quad\text{on }|x|=\phi(t).
\end{equation}

\begin{theorem}\label{thm3.1}
For $\phi$ in a sufficiently small neighborhood of $R$ in $X$, \eqref{eq3.2} has a unique even, $2\pi$-periodic solution $u_\phi\in C^{2,\alpha}(\overline{A_R})$. The pulled-back solution and its inner Neumann trace depend smoothly on $\phi$ and are radial in $x$. At $\phi\equiv R$, the solution is $u_R(x,t)=U_R(|x|)$, where
\begin{equation}\label{eq3.6}
U_R(r)=
\begin{cases}
\dfrac{\gamma}{2n}(1-r^2)+b_R(r^{2-n}-1),&n\geq3,\\[1ex]
\dfrac{\gamma}{4}(1-r^2)+b_R\log r,&n=2,
\end{cases}
\quad
b_R=
\begin{cases}
\dfrac{1+\frac{\gamma}{2n}(R^2-1)}{R^{2-n}-1},&n\geq3,\\[1ex]
\dfrac{1+\frac{\gamma}{4}(R^2-1)}{\log R},&n=2.
\end{cases}
\end{equation}
\end{theorem}
\begin{proof}
Work on the compact quotient $A_R/(2\pi\mathbb Z)$. Uniform ellipticity and the Dirichlet Schauder estimates give existence and uniqueness. The coefficients depend smoothly on $\phi$, so the inverse Dirichlet operator and the solution map are smooth. Rotations in $x$ and reflection in $t$ preserve the equation and the boundary data, and uniqueness gives the asserted symmetries. Solving the radial Poisson equation gives \eqref{eq3.6}. This is the same fixed-domain method used for periodic overdetermined problems in \cite{FallMW}.
\end{proof}

We define the nonlinear map with two independent arguments:
\begin{equation}\label{eq3.7}
H(R,\varphi)(t)=
\partial_{\nu_{R+\varphi}}u_{R+\varphi}(Re_1,t)-Q_R^{\mathrm{cyl}},
\qquad H(R,0)=0,
\end{equation}
where $e_1$ is the first unit vector in $\mathbb R^n$. The Neumann trace is evaluated at $Re_1$, not at $e_1$, on the reference inner boundary. For joint smooth dependence on $(R,\varphi)$ near any $(\rho,0)$, use $A_\rho$ as the single fixed domain and the map
\[
(x,t)\longmapsto\big((1+\eta_\rho(x)(R-\rho+\varphi(t)))x,t\big).
\]
Here $\eta_\rho=1/\rho$ near the inner boundary and vanishes near the outer boundary. Theorem~\ref{thm3.1} applied on this fixed quotient shows that $H$ is smooth into $Y$. The physical boundary trace is independent of the chosen pull-back. Thus $H(R,\varphi)=0$ is equivalent to \eqref{e1.6}.

\subsection{The linearization}
In this subsection $R\in(0,1)$ is fixed, and all the pulled-back functions are defined on $A_R$. For $\psi\in X$, consider
\[
\phi_\varepsilon=R+\varepsilon\psi,
\qquad \Phi_\varepsilon=\Phi_{\phi_\varepsilon},
\qquad g_\varepsilon=g_{\phi_\varepsilon}.
\]
For sufficiently small $|\varepsilon|$, these maps are well defined by the construction in the preceding subsection. Differentiation of the pull-back gives
\begin{equation}\label{eq3.8}
[D_\phi\Phi_\phi]\psi=(\eta(x)x\psi(t),0).
\end{equation}
We denote this vector field by $\mathbf p_\psi$. On the inner boundary, where $\eta=1/R$, it has the value $\mathbf p_\psi(R\theta,t)=(\theta\psi(t),0)$; it vanishes near the outer boundary. Thus $\psi$ describes a displacement in the positive radial direction, whereas the reference outer normal to the fluid region is $\nu=(-x/|x|,0)$ on $|x|=R$.

Extend the radial formula $U_R$ from \eqref{eq3.6} to a collar of $\overline{A_R}$, and regard it as the function $(y,t)\mapsto U_R(|y|)$. Set $u^\phi=U_R\circ\Phi_\phi$. This is the pull-back of the reference radial profile, while $u_\phi$ denotes the pulled-back Dirichlet solution from \eqref{eq3.2}. Since $-\Delta U_R=\gamma$ throughout this collar and $\Phi_\phi$ is an isometry from $(A_R,g_\phi)$ to the physical domain, we have
\begin{equation}\label{eq3.10}
-\Delta_{g_\phi}u^\phi=\gamma.
\end{equation}
In particular, the right-hand side is the same fixed constant as in \eqref{eq3.2}. The two functions need not have the same inner boundary value: on $|x|=R$, $u^\phi=U_R(\phi(t))$, whereas $u_\phi=1$. They both vanish on $|x|=1$.

For clarity, write $\widehat u_\phi=u_\phi\circ\Phi_\phi^{-1}$ for the physical Dirichlet solution. The isometry identity \eqref{eq3.4} gives
\[
\partial_{\nu_\phi}u_\phi(Re_1,t)
=\partial_{\mathbf n_\phi}\widehat u_\phi(\phi(t)e_1,t).
\]
The left-hand side is a metric normal derivative on the fixed inner boundary, and the right-hand side is the Euclidean normal derivative on the moving physical boundary. We differentiate the former at the fixed reference point $(Re_1,t)$.

\begin{lemma}\label{lem3.2}
The map $\phi\mapsto\partial_{\nu_\phi}u^\phi(Re_1,\cdot)$ is smooth into $Y$, and
\begin{equation}\label{eq3.11}
\left[D_\phi\big|_{\phi=R}\,
\partial_{\nu_\phi}u^\phi(Re_1,\cdot)\right]\psi
=-\kappa_R^{\mathrm{cyl}}\psi.
\end{equation}
\end{lemma}
\begin{proof}
At the physical point $(\phi(t)e_1,t)$, the gradient of the radial profile is $(U_R'(\phi(t))e_1,0)$. By \eqref{eq3.4} and \eqref{eq3.5},
\[
\partial_{\nu_\phi}u^\phi(Re_1,t)
=-\frac{U_R'(\phi(t))}{\sqrt{1+\phi'(t)^2}}.
\]
This expression is smooth from a neighborhood of $R$ in $X$ to $Y$. Along $\phi_\varepsilon=R+\varepsilon\psi$, the physical normal is
\[
\mathbf n_{\phi_\varepsilon}(\phi_\varepsilon(t)e_1,t)
=\frac{(-e_1,\varepsilon\psi'(t))}{\sqrt{1+\varepsilon^2\psi'(t)^2}}
=(-e_1,0)+\varepsilon(0,\psi'(t))+O(\varepsilon^2).
\]
The remainder is taken in $C^{1,\alpha}$ as a function of $t$. Its first variation is tangential to the reference cylinder, so it has zero scalar product with the reference radial gradient. Equivalently,
\[
\left.\frac{d}{d\varepsilon}\right|_{\varepsilon=0}
\big(1+\varepsilon^2\psi'(t)^2\big)^{-1/2}=0.
\]
Consequently,
\[
\begin{aligned}
\left.\frac{d}{d\varepsilon}\right|_{\varepsilon=0}
\partial_{\nu_{\phi_\varepsilon}}u^{\phi_\varepsilon}(Re_1,t)
&=\left.\frac{d}{d\varepsilon}\right|_{\varepsilon=0}
\left[-U_R'(R+\varepsilon\psi(t))
\big(1+\varepsilon^2\psi'(t)^2\big)^{-1/2}\right]\\
&=-U_R''(R)\psi(t)
=-\kappa_R^{\mathrm{cyl}}\psi(t).
\end{aligned}
\]
This proves \eqref{eq3.11}, including the case $n=2$.
\end{proof}

\begin{theorem}\label{thm3.3}
The linearized operator $L_R=D_\varphi H(R,0):X\to Y$ is
\begin{equation}\label{eq3.19}
L_R\psi=Q_R^{\mathrm{cyl}}\partial_\nu w_{\psi,R}(Re_1,\cdot)
-\kappa_R^{\mathrm{cyl}}\psi,
\end{equation}
where $\nu=(-x/|x|,0)$ on $|x|=R$ and $w_{\psi,R}$ is the unique even, periodic solution of
\begin{equation}\label{eq3.20}
\left\{
\begin{array}{ll}
\Delta_xw_{\psi,R}+\partial_{tt}w_{\psi,R}=0 &\text{in }A_R,\\
w_{\psi,R}=\psi(t) &\text{on }|x|=R,\\
w_{\psi,R}=0 &\text{on }|x|=1.
\end{array}
\right.
\end{equation}
\end{theorem}
\begin{proof}
Let $a_\phi=u^\phi-u_\phi$. Subtracting the two equations \eqref{eq3.10} and \eqref{eq3.2} cancels their common constant right-hand side. The boundary values described above therefore give
\begin{equation}\label{eq3.21}
-\Delta_{g_\phi}a_\phi=0,\qquad
a_\phi=U_R(\phi(t))-1\ \text{on }|x|=R,\qquad
a_\phi=0\ \text{on }|x|=1.
\end{equation}
At $\phi\equiv R$, both $u^R$ and $u_R$ equal $U_R(|x|)$, so $a_R=0$ throughout $A_R$.

The smooth dependence in Theorem~\ref{thm3.1} allows us to differentiate $a_\phi$ in $C^{2,\alpha}(\overline{A_R})$. Define its material derivative by
\[
b_{\psi,R}=D_\phi a_\phi\big|_{\phi=R}[\psi]
=\left.\frac{d}{d\varepsilon}\right|_{\varepsilon=0}a_{\phi_\varepsilon}.
\]
This derivative is taken with the reference point $(x,t)$ fixed. In general it differs from the Eulerian derivative of the physical function by the transport term. Indeed, if $\widehat a_\phi=U_R-\widehat u_\phi$ and $a'_{\psi,R}$ denotes its derivative at a fixed interior physical point, then
\[
b_{\psi,R}=a'_{\psi,R}
+\nabla\widehat a_R\cdot\mathbf p_\psi.
\]
Here $\widehat a_R=0$, so the transport term vanishes and the two derivatives agree.

To differentiate the equation, put $\mathcal A_\varepsilon=-\Delta_{g_\varepsilon}$. These operators depend smoothly on $\varepsilon$ as bounded maps from $C^{2,\alpha}$ to $C^{0,\alpha}$ on the fixed quotient. Differentiating $\mathcal A_\varepsilon a_{\phi_\varepsilon}=0$ gives
\[
\mathcal A_0b_{\psi,R}
+\left.\frac{d\mathcal A_\varepsilon}{d\varepsilon}\right|_{\varepsilon=0}a_R=0.
\]
The second term is zero because $a_R$ is identically zero. Thus no variation of the coefficients contributes to the equation for $b_{\psi,R}$. By \eqref{eq3.1}, this equation is
\[
\Delta_xb_{\psi,R}+\partial_{tt}b_{\psi,R}=0\qquad\text{in }A_R.
\]
On the fixed inner boundary, differentiation of \eqref{eq3.21} yields
\[
b_{\psi,R}(R\theta,t)
=\left.\frac{d}{d\varepsilon}\right|_{\varepsilon=0}
\big(U_R(R+\varepsilon\psi(t))-1\big)
=U_R'(R)\psi(t)=-Q_R^{\mathrm{cyl}}\psi(t).
\]
The outer boundary value in \eqref{eq3.21} is zero for every $\varepsilon$, so $b_{\psi,R}=0$ on $|x|=1$. Comparison with \eqref{eq3.20} and uniqueness of the Dirichlet problem give
\[
b_{\psi,R}=-Q_R^{\mathrm{cyl}}w_{\psi,R}.
\]

It remains to differentiate the normal trace of $a_\phi$. A metric normal derivative can be written as $\partial_{\nu_\phi}a_\phi=da_\phi(\nu_\phi)$. Hence, at a fixed point of $|x|=R$,
\[
\left.\frac{d}{d\varepsilon}\right|_{\varepsilon=0}
\big(da_{\phi_\varepsilon}(\nu_{\phi_\varepsilon})\big)
=db_{\psi,R}(\nu)
+da_R\left(\left.\frac{d\nu_{\phi_\varepsilon}}{d\varepsilon}\right|_{\varepsilon=0}\right).
\]
The second term, which contains the variation of the pulled-back unit normal, vanishes because $da_R=0$. The variation of the metric is already included in the variation of $\nu_\phi$ in this directional-derivative formula. We obtain
\begin{equation}\label{eq3.24}
D_\phi\big|_{\phi=R}\big(\partial_{\nu_\phi}a_\phi\big)[\psi]
=\partial_\nu b_{\psi,R}.
\end{equation}

Since $u_\phi=u^\phi-a_\phi$, Lemma~\ref{lem3.2} and \eqref{eq3.24}, evaluated at $(Re_1,t)$, now give
\[
\begin{aligned}
L_R\psi
&=-\kappa_R^{\mathrm{cyl}}\psi
-\partial_\nu b_{\psi,R}(Re_1,\cdot)\\
&=Q_R^{\mathrm{cyl}}\partial_\nu w_{\psi,R}(Re_1,\cdot)
-\kappa_R^{\mathrm{cyl}}\psi.
\end{aligned}
\]
The parameter $R$ is held fixed in $D_\varphi H(R,0)$, so the term $Q_R^{\mathrm{cyl}}$ subtracted in \eqref{eq3.7} has no variation here. This proves \eqref{eq3.19}. Equivalently, if
\[
T_R\psi=\partial_\nu w_{\psi,R}(Re_1,\cdot),
\]
then $L_R=Q_R^{\mathrm{cyl}}T_R-\kappa_R^{\mathrm{cyl}}J$, where $J:X\hookrightarrow Y$ is the natural inclusion.
\end{proof}

\subsection{Spectral characteristics of the linearization}
We analyze $L_R$ by taking cosine projections in the periodic variable. For an even $2\pi$-periodic function $\psi$, we use the normalization
\[
\widehat\psi_0=\frac1{2\pi}\int_0^{2\pi}\psi(t)\,dt,
\qquad
\widehat\psi_k=\frac1\pi\int_0^{2\pi}\psi(t)\cos(kt)\,dt
\quad(k\geq1).
\]
The cosine functions are complete in the even subspace of $L^2(0,2\pi)$, and hence
\[
\psi(t)=\widehat\psi_0+\sum_{k=1}^{\infty}\widehat\psi_k\cos(kt)
\quad\text{in }L^2(0,2\pi).
\]
We retain the constant coefficient $\widehat\psi_0$. In the arguments below, we use the individual projections and $L^2$ completeness; we do not require convergence of this series in the H\"older norm.

For an integer $k\geq0$ and $\psi(t)=\cos(kt)$, uniqueness in \eqref{eq3.20} gives a solution of the form
\[
w_{\psi,R}(x,t)=p_k(|x|;R)\cos(kt).
\]
Indeed, writing $r=|x|$, we have
\[
\begin{aligned}
\Delta_x\bigl(p_k(r;R)\cos(kt)\bigr)
&=\left(p_k''+\frac{n-1}{r}p_k'\right)\cos(kt),\\
\partial_{tt}\bigl(p_k(r;R)\cos(kt)\bigr)
&=-k^2p_k\cos(kt).
\end{aligned}
\]
Thus \eqref{eq3.20} reduces to
\begin{equation}\label{eq:cyl-mode}
p_k''+\frac{n-1}{r}p_k'-k^2p_k=0,
\qquad p_k(R;R)=1,\quad p_k(1;R)=0.
\end{equation}
The term $-k^2p_k$ comes from the axial variable $t$ and has no factor $r^{-2}$. Since the outer normal to the fluid region on $r=R$ is $-\partial_r$, we define
\begin{equation}\label{eq:cyl-multiplier}
\mu_k(R)=-p_k'(R;R),\qquad
\Lambda_k(R,\gamma)=Q_R^{\mathrm{cyl}}\mu_k(R)-\kappa_R^{\mathrm{cyl}}.
\end{equation}
Then \eqref{eq3.19} gives
\[
L_R[\cos(kt)]=\Lambda_k(R,\gamma)\cos(kt).
\]
For a general $\psi\in X$, we can also project its harmonic extension in \eqref{eq3.20} directly. The projected function satisfies \eqref{eq:cyl-mode} with inner boundary value $\widehat\psi_k$, so it is $\widehat\psi_kp_k(r;R)$. Taking the inner normal trace therefore gives
\[
\widehat{L_R\psi}_k=\Lambda_k(R,\gamma)\widehat\psi_k
\quad(k\geq0).
\]
These projection identities follow from the $C^{2,\alpha}$ regularity of the extension and do not involve differentiating an infinite Fourier series.

We first solve the constant mode. When $k=0$, equation \eqref{eq:cyl-mode} reads $(r^{n-1}p_0')'=0$. Integrating once gives $p_0'=Cr^{1-n}$. A second integration gives a linear combination of $1$ and $r^{2-n}$ for $n>2$, and of $1$ and $\log r$ for $n=2$. Imposing the two boundary values yields
\[
p_0(r;R)=
\begin{cases}
\dfrac{r^{2-n}-1}{R^{2-n}-1},&n\geq3,\\[1ex]
\dfrac{\log r}{\log R},&n=2,
\end{cases}
\qquad
\mu_0(R)=
\begin{cases}
\dfrac{n-2}{R(1-R^{n-2})},&n\geq3,\\[1ex]
\dfrac1{R\log(1/R)},&n=2.
\end{cases}
\]
In both cases $p_0$ decreases from $1$ to $0$, and $\mu_0(R)>0$.

For $k\geq1$, we put $\beta=(n-2)/2$, so that $n-1=2\beta+1$, and make the substitution
\[
s=kr,\qquad p(r)=r^{-\beta}z(s).
\]
Here primes on $z$ denote derivatives with respect to $s$. Direct calculation gives
\[
\begin{aligned}
p'(r)&=r^{-\beta-1}\bigl(-\beta z+s z'\bigr),\\
p''(r)&=r^{-\beta-2}\bigl(\beta(\beta+1)z-2\beta s z'+s^2z''\bigr).
\end{aligned}
\]
Consequently,
\[
p''+\frac{n-1}{r}p'-k^2p
=r^{-\beta-2}\bigl(s^2z''+s z'-(s^2+\beta^2)z\bigr).
\]
The equation for $z$ is therefore the modified Bessel equation
\[
s^2z''+s z'-(s^2+\beta^2)z=0.
\]
Its independent solutions on $s>0$ are the modified Bessel functions $I_\beta$ and $K_\beta$. Thus every solution of the radial equation is a linear combination of $r^{-\beta}I_\beta(kr)$ and $r^{-\beta}K_\beta(kr)$. The condition at $r=1$ selects the combination below, and division by its value at $r=R$ imposes the inner boundary condition:
\begin{equation}\label{eq:cyl-bessel}
Z_k(r)=r^{-\beta}\big(K_\beta(kr)I_\beta(k)-I_\beta(kr)K_\beta(k)\big),\qquad
p_k(r;R)=\frac{Z_k(r)}{Z_k(R)}.
\end{equation}

We verify that this division is well defined and record the sign of the derivative. For $s>0$, the Wronskian identity is
\[
I_\beta(s)K_\beta'(s)-I_\beta'(s)K_\beta(s)=-\frac1s;
\]
see \cite[Section 10.28]{DLMF}. Since $I_\beta(s)>0$ for $\beta\geq0$, it follows that
\[
\frac{d}{ds}\left(\frac{K_\beta(s)}{I_\beta(s)}\right)
=-\frac1{s I_\beta(s)^2}<0.
\]
For $0<r<1$, we have $kr<k$, and hence
\[
\begin{aligned}
Z_k(r)
&=r^{-\beta}I_\beta(kr)I_\beta(k)
\left(\frac{K_\beta(kr)}{I_\beta(kr)}
-\frac{K_\beta(k)}{I_\beta(k)}\right)>0.
\end{aligned}
\]
Thus $Z_k$ has no zero in $(0,1)$, while $Z_k(1)=0$. Differentiating \eqref{eq:cyl-bessel} at $r=1$, the derivative of $r^{-\beta}$ contributes zero because the bracket vanishes there. The Wronskian gives
\[
Z_k'(1)
=k\bigl(K_\beta'(k)I_\beta(k)-I_\beta'(k)K_\beta(k)\bigr)=-1.
\]
The zero at $1$ is therefore simple. In particular, \eqref{eq:cyl-bessel} satisfies both boundary values in \eqref{eq:cyl-mode}. The maximum principle applied on $(R,1)$ also gives $0<p_k<1$.

The derivative recurrences for the modified Bessel functions give
\[
\frac{d}{dr}\bigl(r^{-\beta}I_\beta(kr)\bigr)
=k r^{-\beta}I_{\beta+1}(kr),\qquad
\frac{d}{dr}\bigl(r^{-\beta}K_\beta(kr)\bigr)
=-k r^{-\beta}K_{\beta+1}(kr);
\]
see \cite[Section 10.29]{DLMF}. Therefore
\[
Z_k'(r)=-k r^{-\beta}
\bigl(K_{\beta+1}(kr)I_\beta(k)+I_{\beta+1}(kr)K_\beta(k)\bigr)<0,
\]
and the multipliers have the explicit expression
\[
\mu_k(R)
=k\frac{K_{\beta+1}(kR)I_\beta(k)+I_{\beta+1}(kR)K_\beta(k)}
{K_\beta(kR)I_\beta(k)-I_\beta(kR)K_\beta(k)}
\quad(k\geq1).
\]
The denominator is positive by the preceding ratio argument. These formulas include $n=2$, for which $\beta=0$ and the derivative recurrences involve $I_1$ and $K_1$.

We can also describe the behavior at high frequencies for a fixed $R\in(0,1)$. The large-argument expansions in \cite[Section 10.40]{DLMF} give, with the order $\nu$ fixed,
\[
I_\nu(s)=\frac{e^s}{\sqrt{2\pi s}}\bigl(1+O(s^{-1})\bigr),\qquad
K_\nu(s)=\sqrt{\frac\pi{2s}}e^{-s}
\left(1+\frac{4\nu^2-1}{8s}+O(s^{-2})\right).
\]
To see the effect of the outer boundary, write
\[
q_k=\frac{K_\beta(k)}{I_\beta(k)},\qquad
\theta_k(R)=q_k\frac{I_\beta(kR)}{K_\beta(kR)}\in(0,1).
\]
Dividing the formula for $\mu_k$ by $I_\beta(k)$ and using the Wronskian gives the exact identity
\[
\mu_k(R)
=k\frac{K_{\beta+1}(kR)}{K_\beta(kR)}
+\frac{q_k}{R K_\beta(kR)^2(1-\theta_k(R))}.
\]
For fixed $R$, the expansions imply
\[
\theta_k(R)=O\bigl(e^{-2k(1-R)}\bigr),\qquad
\frac{q_k}{R K_\beta(kR)^2(1-\theta_k(R))}
=O\bigl(k e^{-2k(1-R)}\bigr),
\]
and
\[
\frac{K_{\beta+1}(kR)}{K_\beta(kR)}
=1+\frac{2\beta+1}{2kR}+O(k^{-2}).
\]
We obtain
\[
\mu_k(R)=k+\frac{n-1}{2R}+O(k^{-1})
\quad\text{as }k\to\infty.
\]
The constants here may depend on $R$ and $n$. This asymptotic describes the first-order growth of the Dirichlet-to-Neumann multipliers. We establish the required isomorphism between the H\"older spaces by the mixed boundary problem below.

\begin{theorem}\label{thm3.4}
For every $n\geq2$, there is $\gamma_{\mathrm{cyl}}(n)>0$ and a smooth radius $R_*=R_*(n,\gamma)\in(0,1)$ for $|\gamma|<\gamma_{\mathrm{cyl}}(n)$ such that
\begin{equation}\label{eq:cyl-simple-zero}
\Lambda_1(R_*,\gamma)=0,\qquad
\partial_R\Lambda_1(R_*,\gamma)\ne0,\qquad
Q_{R_*}^{\mathrm{cyl}}>0.
\end{equation}
Moreover, $\Lambda_0(R_*,\gamma)<0$ and $\Lambda_k(R_*,\gamma)>0$ for all $k\geq2$.
\end{theorem}

\begin{lemma}\label{lem:cyl-modes}
For every $R\in(0,1)$ and $0\leq k<\ell$, $\mu_\ell(R)>\mu_k(R)$.
\end{lemma}
\begin{proof}
We compare the normalized solutions $p_k$ and $p_\ell$ at the same radius $R$. The maximum principle gives $0<p_j<1$ on $(R,1)$ for $j=k,\ell$, including $j=0$. In divergence form, their equations are
\[
(r^{n-1}p_j')'=j^2r^{n-1}p_j.
\]
Multiplying the equation for $p_\ell$ by $p_k$, multiplying the equation for $p_k$ by $p_\ell$, and subtracting, we obtain
\[
\frac{d}{dr}\left[r^{n-1}(p_\ell'p_k-p_k'p_\ell)\right]
=(\ell^2-k^2)r^{n-1}p_kp_\ell.
\]
At $r=1$, the expression in brackets vanishes because both functions are zero. At $r=R$, both are $1$, so this expression equals $R^{n-1}(\mu_k-\mu_\ell)$. Integration from $R$ to $1$ therefore yields
\[
R^{n-1}\big(\mu_\ell(R)-\mu_k(R)\big)
=(\ell^2-k^2)\int_R^1r^{n-1}p_kp_\ell\,dr>0.
\]
This proves the strict comparison for every pair of modes, including the constant mode.
\end{proof}

\begin{lemma}\label{lem:cyl-radius}
For each $k\geq1$, $M_k(R):=R\mu_k(R)$ is strictly increasing on $(0,1)$, and
\[
\lim_{R\to0^+}M_k(R)=n-2,\qquad
\lim_{R\to1^-}M_k(R)=+\infty.
\]
The first limit equals zero when $n=2$.
\end{lemma}
\begin{proof}
We fix $k\geq1$. The function $Z_k$ in \eqref{eq:cyl-bessel} is independent of the inner radius $R$, so
\[
\mu_k(R)=-\frac{Z_k'(R)}{Z_k(R)}.
\]
We already know that $Z_k>0$ and $Z_k'<0$ on $(0,1)$. The derivative sign can also be seen directly from its equation:
\[
(r^{n-1}Z_k')'=k^2r^{n-1}Z_k>0,\qquad
r^{n-1}Z_k'(r)
=-1-k^2\int_r^1s^{n-1}Z_k(s)\,ds<0.
\]
We set
\[
y(r)=-\frac{rZ_k'(r)}{Z_k(r)}>0,
\qquad M_k(R)=y(R).
\]
Differentiating this quotient and using $Z_k''=k^2Z_k-(n-1)Z_k'/r$, we find
\[
\begin{aligned}
ry'
&=-\frac{rZ_k'}{Z_k}-\frac{r^2Z_k''}{Z_k}
+r^2\left(\frac{Z_k'}{Z_k}\right)^2\\
&=y+(n-1)\frac{rZ_k'}{Z_k}-k^2r^2+y^2.
\end{aligned}
\]
Thus $y$ satisfies the Riccati equation
\begin{equation}\label{eq:cyl-riccati}
ry'=y^2-(n-2)y-k^2r^2.
\end{equation}
The positive root of the polynomial on the right is
\[
c(r)=\frac{n-2+\sqrt{(n-2)^2+4k^2r^2}}2,
\qquad
c'(r)=\frac{2k^2r}{\sqrt{(n-2)^2+4k^2r^2}}>0.
\]
The other root is negative for $r>0$. Hence, because $y>0$, the right-hand side of \eqref{eq:cyl-riccati} is positive exactly when $y>c$. For $n=2$ we have $c(r)=kr$ and $c'(r)=k>0$.

We next prove that $y(r)>c(r)$ for all $0<r<1$. Since $Z_k(1)=0$ and $Z_k'(1)=-1$, Taylor expansion at $1$ gives
\[
Z_k(r)=(1-r)+O((1-r)^2),\qquad
Z_k'(r)=-1+O(1-r).
\]
It follows that $y(r)\to+\infty$ as $r\to1^-$, whereas $c(r)$ remains bounded there. Thus $y-c>0$ on a neighborhood of $1$. If the assertion fails, choose $a\in(0,1)$ with $y(a)-c(a)\leq0$ and then $b\in(a,1)$ such that $y-c>0$ on $[b,1)$. Continuity gives a nonempty closed set of contact points in $[a,b]$. Let $r_0$ be its largest point. There are no contact points to its right, and the positivity near $1$ gives $y-c>0$ on $(r_0,1)$. Therefore
\[
(y-c)'(r_0)\geq0.
\]
But $y(r_0)=c(r_0)$ makes the right-hand side of \eqref{eq:cyl-riccati} zero, so $y'(r_0)=0$ and
\[
(y-c)'(r_0)=-c'(r_0)<0,
\]
a contradiction. We conclude that $y>c$ and hence $y'>0$ on $(0,1)$. This proves the strict increase of $M_k$.

We now check its limit at zero. When $n>2$, we have $\beta>0$, and the small-argument forms in \cite[Section 10.30]{DLMF} are
\[
I_\beta(s)\sim\frac{s^\beta}{2^\beta\Gamma(\beta+1)},
\qquad
K_\beta(s)\sim2^{\beta-1}\Gamma(\beta)s^{-\beta}.
\]
In \eqref{eq:cyl-bessel}, the $K_\beta(kr)$ term is the leading term. Thus
\[
Z_k(r)\sim C_k r^{-2\beta}=C_k r^{2-n},
\qquad
C_k=2^{\beta-1}\Gamma(\beta)k^{-\beta}I_\beta(k)>0.
\]
To obtain the derivative asymptotic, we use the exact derivative formula above, together with
\[
K_{\beta+1}(s)\sim2^\beta\Gamma(\beta+1)s^{-\beta-1},
\qquad
I_{\beta+1}(s)\sim\frac{s^{\beta+1}}{2^{\beta+1}\Gamma(\beta+2)}.
\]
This gives
\[
Z_k'(r)\sim-2\beta C_k r^{-2\beta-1}
=-(n-2)C_k r^{1-n},
\qquad
-\frac{rZ_k'(r)}{Z_k(r)}\longrightarrow n-2.
\]
In this step, the derivative asymptotic follows from the Bessel derivative identities, rather than from differentiating an asymptotic equivalence.

For $n=2$, we have $\beta=0$. The corresponding forms are
\[
I_0(s)\sim1,\qquad K_0(s)\sim\log(1/s),
\qquad I_1(s)\sim\frac{s}{2},\qquad K_1(s)\sim\frac1s.
\]
Using \eqref{eq:cyl-bessel} and its derivative formula with $\beta=0$, we obtain
\[
Z_k(r)\sim I_0(k)\log(1/r),\qquad
Z_k'(r)\sim-\frac{I_0(k)}r,
\qquad
-\frac{rZ_k'(r)}{Z_k(r)}\sim\frac1{\log(1/r)}\longrightarrow0.
\]
This establishes the zero-end limit also in dimension $n=2$. Finally, the Taylor expansion at $1$ gives $y(r)\sim r/(1-r)$, so the other limit is $+\infty$.
\end{proof}

\begin{proof}[Proof of Theorem~\ref{thm3.4}]
We first take $\gamma=0$. Then $U_R=p_0(\cdot;R)$, so $Q_R^{\mathrm{cyl}}=\mu_0(R)>0$, and
\[
\frac{\kappa_R^{\mathrm{cyl}}}{Q_R^{\mathrm{cyl}}}
=\frac{n-1}{R},\qquad
\Lambda_1(R,0)=\frac{Q_R^{\mathrm{cyl}}}{R}\big(M_1(R)-(n-1)\big).
\]
By Lemma~\ref{lem:cyl-radius}, $M_1$ increases strictly from $n-2$ to $+\infty$. Since $n-2<n-1$, there is exactly one $R_0\in(0,1)$ with $M_1(R_0)=n-1$. Substitution of this value into \eqref{eq:cyl-riccati}, with $k=1$, gives
\[
M_1'(R_0)
=\frac{(n-1)^2-(n-2)(n-1)-R_0^2}{R_0}
=\frac{n-1-R_0^2}{R_0}>0.
\]
The last inequality follows from $n\geq2$ and $R_0<1$; in particular, it is strict when $n=2$. When we differentiate the formula for $\Lambda_1(R,0)$, the derivative of $Q_R^{\mathrm{cyl}}/R$ is multiplied by $M_1(R)-(n-1)$ and therefore contributes zero at the root. Consequently,
\[
\partial_R\Lambda_1(R_0,0)
=\frac{Q_{R_0}^{\mathrm{cyl}}}{R_0}M_1'(R_0)>0.
\]

The Bessel quotient is smooth for $0<R<1$, and the formulas in \eqref{eq3.6} show that $Q_R^{\mathrm{cyl}}$ and $\kappa_R^{\mathrm{cyl}}$ are smooth in $(R,\gamma)$. Thus $(R,\gamma)\mapsto\Lambda_1(R,\gamma)$ is smooth near $(R_0,0)$. The implicit function theorem gives a smooth root $R_*=R_*(n,\gamma)$ near $R_0$ for sufficiently small $|\gamma|$. We shrink $\gamma_{\mathrm{cyl}}(n)>0$ so that $R_*\in(0,1)$, $Q_{R_*}^{\mathrm{cyl}}>0$, and $\partial_R\Lambda_1(R_*,\gamma)>0$ throughout this interval.

At the root, the equation $\Lambda_1(R_*,\gamma)=0$ is exactly $\kappa_{R_*}^{\mathrm{cyl}}=Q_{R_*}^{\mathrm{cyl}}\mu_1(R_*)$. Hence for every $k\geq0$,
\[
\Lambda_k(R_*,\gamma)
=Q_{R_*}^{\mathrm{cyl}}\big(\mu_k(R_*)-\mu_1(R_*)\big).
\]
Lemma~\ref{lem:cyl-modes} gives $\mu_0<\mu_1$ and $\mu_k>\mu_1$ for every $k\geq2$. Since $Q_{R_*}^{\mathrm{cyl}}>0$, these inequalities prove the stated signs, including the negative constant multiplier. This exact comparison establishes nonresonance for all modes at once, with the same interval of allowed $\gamma$.
\end{proof}

\begin{lemma}\label{lem:cyl-fredholm}
At $R_*$, $L_{R_*}:X\to Y$ is Fredholm of index zero, with
\[
\ker L_{R_*}=\operatorname{span}\{\cos t\},\qquad
\operatorname{Range}L_{R_*}
=\left\{h\in Y:\int_0^{2\pi}h(t)\cos t\,dt=0\right\}.
\]
Moreover,
\[
D_R D_\varphi H(R_*,0)[\cos t]
\notin\operatorname{Range}L_{R_*}.
\]
\end{lemma}
\begin{proof}
We write $T_R:X\to Y$ for the inner Dirichlet-to-Neumann operator in \eqref{eq3.20}, namely $T_R\psi=\partial_\nu w_{\psi,R}(Re_1,\cdot)$. The Dirichlet Schauder estimate and the boundary trace estimate make $T_R$ bounded. We first show that it is an isomorphism. On the compact quotient $A_R/(2\pi\mathbb Z)$, consider
\[
\Delta w=0,\qquad w=0\ \text{on }|x|=1,\qquad
\partial_\nu w=h(t)\ \text{on }|x|=R.
\]
The outer Dirichlet condition removes the constant ambiguity and gives a coercive mixed boundary problem. Its two boundary components are disjoint and smooth, so the mixed boundary Schauder estimates yield a unique solution with
\[
\|w\|_{C^{2,\alpha}}\leq C_R\|h\|_{C^{1,\alpha}}
\quad(h\in Y);
\]
see \cite{Gilbarg}. Uniqueness makes this solution radial in $x$ and even in $t$. Its inner trace $\psi(t)=w(Re_1,t)$ belongs to $X$ and satisfies $T_R\psi=h$, with a bounded trace estimate. For injectivity, if $T_R\psi=0$, Green's identity for its harmonic extension gives $\int|\nabla w_{\psi,R}|^2=0$. The extension is constant and vanishes on the outer boundary, so it is zero and $\psi=0$. This proves that $T_R^{-1}:Y\to X$ is bounded.

By \eqref{eq3.19},
\[
L_R=Q_R^{\mathrm{cyl}}T_R-\kappa_R^{\mathrm{cyl}}J,
\]
where the inclusion $J:X\hookrightarrow Y$ is compact. At $R_*$, the scalar $Q_{R_*}^{\mathrm{cyl}}$ is positive, so $Q_{R_*}^{\mathrm{cyl}}T_{R_*}$ is an isomorphism. Equivalently,
\[
L_{R_*}=Q_{R_*}^{\mathrm{cyl}}T_{R_*}
\left(I-\frac{\kappa_{R_*}^{\mathrm{cyl}}}{Q_{R_*}^{\mathrm{cyl}}}
T_{R_*}^{-1}J\right).
\]
The operator in parentheses is the identity minus a compact operator on $X$. Therefore $L_{R_*}$ is Fredholm of index zero and its range is closed.

The cosine projection identities established above show that $L_{R_*}\psi=0$ forces
\[
\Lambda_k(R_*,\gamma)\widehat\psi_k=0
\quad\text{for every }k\geq0.
\]
By Theorem~\ref{thm3.4}, only $k=1$ has zero multiplier. Thus every coefficient other than $\widehat\psi_1$ vanishes. Completeness in $L^2$ gives $\psi=\widehat\psi_1\cos t$, and continuity gives the same identity pointwise. Conversely, $L_{R_*}\cos t=0$. Hence the kernel is exactly $\operatorname{span}\{\cos t\}$, and the index-zero property gives a range of codimension one.

To identify this range, we use Green's identity. For $\psi,\zeta\in X$, let $w_\psi$ and $w_\zeta$ be their harmonic extensions with zero outer boundary values. On the compact quotient,
\[
\begin{aligned}
\int_{A_R/(2\pi\mathbb Z)}\nabla w_\psi\cdot\nabla w_\zeta
&=\int_{\partial(A_R/(2\pi\mathbb Z))}w_\zeta\partial_\nu w_\psi\\
&=|\mathbb S^{n-1}|R^{n-1}
\int_0^{2\pi}\zeta(t)(T_R\psi)(t)\,dt.
\end{aligned}
\]
The outer boundary term vanishes. Interchanging $\psi$ and $\zeta$ leaves the volume integral unchanged; cancelling the common inner surface factor gives
\[
\int_0^{2\pi}(T_R\psi)\zeta\,dt
=\int_0^{2\pi}\psi(T_R\zeta)\,dt.
\]
The scalar multiplication term in $L_R$ has the same symmetry. Therefore $L_{R_*}\cos t=0$ implies
\[
\int_0^{2\pi}(L_{R_*}\psi)(t)\cos t\,dt=0
\quad\text{for every }\psi\in X.
\]
The hyperplane defined by this integral has codimension one in $Y$. Since the closed range of $L_{R_*}$ also has codimension one and is contained in that hyperplane, the two spaces coincide.

Finally, the Banach spaces $X$ and $Y$ use the same periodic variable for all $R$, and $H$ depends smoothly on $R$. We can differentiate the identity $L_R\cos t=\Lambda_1(R,\gamma)\cos t$ with respect to $R$, keeping $\gamma$ fixed. This gives
\[
D_R D_\varphi H(R_*,0)[\cos t]
=\partial_R\Lambda_1(R_*,\gamma)\cos t.
\]
Its pairing with $\cos t$ is
\[
\int_0^{2\pi}\partial_R\Lambda_1(R_*,\gamma)\cos^2t\,dt
=\pi\partial_R\Lambda_1(R_*,\gamma)\ne0.
\]
It therefore lies outside the range just identified. This proves transversality.
\end{proof}

\subsection{The Proof of Theorem \ref{thm1.2}}
\begin{proof}[Proof of Theorem~\ref{thm1.2}]
Fix $|\gamma|<\gamma_{\mathrm{cyl}}(n)$. Choose an open neighborhood $\mathcal U\subset(0,1)\times X$ of $(R_*,0)$ small enough that $0<R+\varphi(t)<1$ for every $(R,\varphi)\in\mathcal U$ and that the fixed-domain construction applies. The map $H:\mathcal U\to Y$ is smooth, and $H(R,0)=0$ for all $R$ sufficiently close to $R_*$. By Lemma~\ref{lem:cyl-fredholm}, its derivative in $\varphi$ has a one-dimensional kernel, a codimension-one range, and satisfies the transversality condition. The Crandall--Rabinowitz theorem \cite[Theorem 1.7]{Crandall} gives a smooth curve $(R_s,\varphi_s)$ of zeros with
\[
R_0=R_*,\qquad
\varphi_s=s\cos t+s\mu_s,\qquad
\mu_0=0,\qquad
\int_0^{2\pi}\mu_s(t)\cos t\,dt=0.
\]
Here we use the closed complement of $\operatorname{span}\{\cos t\}$ defined by this integral condition. Smoothness of the curve follows from smoothness of $H$ and the implicit function theorem in the usual simple-eigenvalue reduction.

After reducing $\varepsilon$, $0<R_s+\varphi_s<1$ and $Q_{R_s}^{\mathrm{cyl}}>0$. Returning from the fixed quotient to the physical periodic domain yields the solution of \eqref{e1.6}. For $s\ne0$,
\[
\int_0^{2\pi}\phi_s(t)\cos t\,dt=\pi s\ne0,
\]
so the inner radius is nonconstant. This proves the claimed nontriviality and completes the theorem.
\end{proof}

\bibliographystyle{amsplain}
\makeatletter
\def\@biblabel#1{#1.~}
\makeatother


\providecommand{\bysame}{\leavevmode\hbox to3em{\hrulefill}\thinspace}
\providecommand{\MR}{\relax\ifhmode\unskip\space\fi MR }
\providecommand{\MRhref}[2]{%
  \href{http://www.ams.org/mathscinet-getitem?mr=#1}{#2}
}
\providecommand{\href}[2]{#2}

\end{document}